\RequirePackage{fix-cm}
\RequirePackage{amsmath}

\documentclass[12pt]{elsarticle}

\usepackage[utf8]{inputenc}
\usepackage{amssymb}
\usepackage{amsmath}
\usepackage{amsthm}
\usepackage{natbib}
\usepackage{setspace}
\usepackage{algorithm}
\usepackage{url}
\usepackage{subfigure}
\usepackage{color}
\usepackage{morefloats}
\usepackage{longtable}
\usepackage{enumerate}
\usepackage{mathrsfs,bm}
\usepackage{textcomp}
\usepackage{graphicx}
\usepackage{multirow}
\usepackage{latexsym}
\usepackage[figuresright]{rotating}
\usepackage{array}
\usepackage{tikz}

\usepackage[noend]{algpseudocode}

\newtheorem{proposition}{Proposition}

\newtheorem{theorem}{Theorem}
\newtheorem{claim}{Claim}

\newtheorem{lemma}{Lemma}

\begin{document}

\begin{frontmatter}


\title{Convex geometries over induced paths with bounded length}

\author[unlp]{Marisa Gutierrez}
\ead{marisa@mate.unlp.edu.ar}

\author[uff]{F\'abio Protti\corref{cor1}}
\ead{fabio@ic.uff.br}

\author[unlp]{Silvia B. Tondato}
\ead{tondato@mate.unlp.edu.ar}

\cortext[cor1]{Corresponding author}

\address[unlp]{Departamento de Matem\'atica, Facultad de Ciencias Exactas\\ Universidad Nacional de La Plata, Argentina\\
\hspace*{1cm}}
\address[uff]{Instituto de Computa\c c\~ao\\
Universidade Federal Fluminense, Niter\'oi, Brazil}


\begin{abstract}

Graph convexity spaces have been studied in many contexts. In particular, some studies are devoted to determine if a graph equipped with a convexity space is a {\em convex geometry}. It is well known that chordal and Ptolemaic graphs can be characterized as convex geometries with respect to the geodesic and monophonic convexities, respectively. Weak polarizable graphs, interval graphs, and proper interval graphs can also be characterized in this way. In this paper we introduce the notion of {\em $l^k$-convexity}, a natural restriction of the monophonic convexity. Let $G$ be a graph and $k\geq 2$ an integer. A subset $S\subseteq V(G)$ is \textit{$l^k$-convex} if and only if for any pair of vertices $x,y$ of $S$, each induced path of length {\em at most} $k$ connecting $x$ and $y$ is completely contained in the subgraph induced by $S$. The {\em $l^k$-convexity} consists of all $l^k$-convex subsets of $G$. In this work, we characterize {\em $l^k$-convex geometries} (graphs that are convex geometries with respect to the $l^k$-convexity) for $k\in\{2,3\}$. We show that a graph $G$ is an $l^2$-convex geometry if and only if $G$ is a chordal $P_4$-free graph, and an $l^3$-convex geometry if and only if $G$ is a chordal graph with diameter at most three such that its induced gems satisfy a special ``solving'' property. As far as the authors know, the class of $l^3$-convex geometries is the first example of a non-hereditary class of convex geometries.

\end{abstract}


\end{frontmatter}

\journal{Discrete Applied Mathematics}

\section{Introduction}
\label{sec:intro}

A family $\mathcal{C}$ of subsets of a nonempty set $V$ is called a \textit{convexity} on $V$ if:

\begin{itemize}
\item $(C1)$ $\emptyset, V \in \mathcal{C}$.

\item $(C2)$ $\mathcal{C}$ is stable for intersections, that is, if $\mathcal{D}$ is a non-empty subfamily of $\mathcal{C}$ then $\bigcap \mathcal{D}$ is a member of $\mathcal{C}$.

\item $(C3)$ $\mathcal{C}$ is stable for nested unions, that is, if $\mathcal{D}$ is a non-empty and totally ordered by inclusion subfamily of $\mathcal{C}$ then $\bigcup \mathcal{D}$ is a member of $\mathcal{C}$.

\end{itemize}

A \textit{convexity space} is an ordered pair $(V, \mathcal{C})$, where $V$ is a nonempty set and $\mathcal{C}$ is a convexity on $V$. The members of $\mathcal{C}$ are called \textit{convex sets}. A \textit{graph convexity space} is an ordered pair $(G, \mathcal{C})$ formed by a connected graph $G$ and a convexity $\mathcal{C}$ on $V(G)$ such that $(V(G),\mathcal{C})$ is a convexity space.





In the last few decades, convexity spaces and graph convexity spaces have been studied in many contexts~\cite{farber-jamison,pelayo,van-de-vel}. In particular, some studies are devoted to determine if a graph equipped with a convexity space is a {\em convex geometry}. The concept of convex geometry is related to the concepts of {\em convex hull} and {\em extreme point}. We refer the reader to~\cite{farber-jamison}. Let $(V, \mathcal{C})$ be a convexity space. Given a set $S \subseteq V$, the smallest convex set containing $S$ is called the \textit{convex hull} of $S$. An element $x$ of a convex set $S$ is an \textit{extreme point} of $S$ if $S\backslash\{x\}$ is also convex. The convexity space $(V, \mathcal{C})$ is said to be a \textit{convex geometry} if it satisfies the so-called \textit{Minkowski-Krein-Milman} property~\cite{krein-milman}:

\medskip

\centerline{{\em Every convex set is the convex hull of its extreme points.}}

\medskip

Let $(G,\mathcal{C})$ be a graph convexity space. There are many examples in the literature where the convexity $\mathcal{C}$ is defined over a path system. For example, the {\em monophonic convexity}~\cite{dourado-et-al,duchet} consists of all the monophonically convex sets of $V(G)$ (a set $S$ is {\em monophonically convex} if and only if every {\em induced} path between two vertices of $S$ lies entirely in the subgraph induced by $S$). Likewise, the {\em geodesic}, $m^3$-, {\em toll}, and {\em weakly toll} convexities are defined over shortest paths, induced paths of length at least three, tolled walks~\cite{alcon-et-al}, and weakly tolled walks~\cite{gutierrez-tondato}.

Chordal and Ptolemaic graphs have been characterized as convex geometries with respect to the monophonic convexity and the geodesic convexity, respectively \cite{farber-jamison}. Similarly, weak polarizable graphs~\cite{olariu} have been characterized  as convex geometries with respect to the $m^3$-convexity~\cite{dragan-et-al}; interval graphs have been characterized as convex geometries with respect to the toll convexity~\cite{alcon-et-al}; and proper interval graphs have been characterized as convex geometries with respect to the weakly toll convexity~\cite{gutierrez-tondato}.

All the above-mentioned classes are hereditary for induced subgraphs. The natural question that arises is whether every convex geometry with respect to a path system defines a hereditary class of graphs. In this work we answer negatively to this question.

Inspired by the studies of Dragan, Nicolai, and Brandst\"adt in~\cite{dragan-et-al},  in this paper we introduce the notion of $l^k$-convexity. Let $k\geq 2$ be an integer. A subset $S\subseteq V(G)$ is called \textit{$l^k$-convex} if and only if for any pair of vertices $x,y$ of $S$, each induced path of length {\em at most} $k$ connecting $x$ and $y$ is completely contained in the subgraph induced by $S$. The {\em $l^k$-convexity} is the convexity consisting of all the $l^k$-convex sets of a graph $G$. If $G$ is a convex geometry with respect to the $l^k$-convexity, we say that $G$ is an {\em $l^k$-convex geometry}.

The main contribution of this work is to provide characterizations of $l^k$-convex geometries for $k\in\{2,3\}$. We show that a graph $G$ is an $l^2$-convex geometry if and only if $G$ is trivially perfect, or, equivalently, a chordal $P_4$-free graph~\cite{golumbic}. We also show that $G$ is an $l^3$-convex geometry if and only if $G$ is chordal, $\mathit{diam}(G)\leq 3$, and its induced gems with at least six vertices satisfy a special ``solving'' property, in the sense that there must be some external structures preventing such gems from being obstacles for $G$ to be an $l^3$-convex geometry. Interestingly, we show that $l^3$-convex geometries do not form a hereditary class of graphs. As far as the authors know, this is the first example of a non-hereditary class of convex geometries.

The paper is organized as follows. Section 2 contains the necessary background. In Section 3, we prove that $l^2$-convex geometries are precisely the chordal $P_4$-free graphs. In Section 4, we describe a characterization of the class of $l^3$-convex geometries and show that such a class is not hereditary. Section 5 contains our conclusions.

\section{Preliminaries}
\label{sec:prelim}

All the graphs in this paper are finite, undirected, simple, and connected. Let $G$ be a graph. An \textit{induced path} $P$ in a $G$ is a sequence of vertices $x_0,\ldots,x_p$ such that $x_ix_j \in E(G)$ if and only if $i=j-1$, $j=1,\ldots,p$. The length of $P$ is $p$, which is the number of edges of $P$. We denote by $P_n$ the induced path with $n$ vertices. An \textit{induced cycle} $C$ is a sequence of vertices $x_0,\ldots,x_p$ such that: (i) $x_0=x_p$ and (ii) $x_ix_j \in E(G)$ if and only if $\{i,j\}=\{0,p-1\}$ or $|i-j|=1$. We denote by $C_n$ the induced cycle with $n$ vertices.

If $P=x_0,\ldots,x_p$ is a path, $P[x_i,x_j] \ (0\leq i\leq j\leq p)$ denotes the path $P'= x_i,x_{i+1},\ldots,x_{j-1},x_j$.

The \textit{distance} $d_G(u,v)$ between two vertices $u,v$ is the minimum number of edges in a path connecting these vertices. The \textit{diameter} of $G$, denoted by $\mathit{diam}(G)$, is the maximum distance between two vertices of $G$.

The neighborhood (resp., closed neighborhood) of $x\in V(G)$ is denoted by $N(x)$ (resp., $N[x]$).

Let $S\subseteq V(G)$. We denote by $G[S]$ the subgraph of $G$ induced by $S$. A \emph{clique} in $G$ is a set of pairwise adjacent vertices. We denote by $\mathcal{C}(G)$ the family of all maximal cliques of $G$. A vertex $x\in V(G)$ is a {\em simplicial vertex} if $N[x]$ is a clique.

Let $xy$ be an edge of $G$ and $z,w$ be two nonadjacent vertices of $G$. The graph $G-xy+zw$ is obtained from $G$ by deleting the edge $xy$ and adding the edge $zw$.

We say that $G$ {\em contains} a graph $H$ if $H$ is an induced subgraph of $G$. In addition, $G$ is {\em $H$-free} if $G$ does not contain $H$.

A vertex $u$ is a {\em universal vertex} if $u$ is adjacent to every other vertex of the graph. A \textit{gem} is a graph $G_n$ such that: (i)  $V(G_n)=\{x_0,\ldots,x_n,u_n\}$ ($n\geq 3$); (ii) $x_0,\ldots,x_n$ is an induced path; (iii) $u_n$ is a universal vertex. We also say that $G_n$ is an {$n$-gem}, to mean that the induced path $x_0,x_1,\ldots,x_n$ contains $n$ edges. See Figure~\ref{fig0}.

\begin{figure}[ht]
\tikzstyle{miEstilo}= [thin, dotted]
\begin{center}
\begin{tikzpicture}[scale=0.8]
\node[draw,circle] (1) at (-4,0) {$x_0$};
 \node[draw,circle] (2) at (-2,0) {$x_1$};
  \node[draw,circle] (3) at (2,0) {$x_n$};
   \node[draw,circle] (4) at (0,2) {$u_n$};
   \draw[miEstilo]  (2)--(3);
   \draw (1)--(2);
   \draw (2)--(4)--(3);
   \draw (4)--(1);
   \draw[miEstilo] (4)--(-1,0);
    \draw[miEstilo] (4)--(0,0);
     \draw[miEstilo] (4)--(1,0);
\end{tikzpicture}\end{center}
\caption{An $n$-gem.} \label{fig0}
\end{figure}
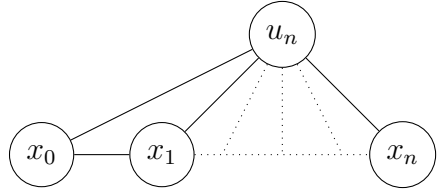

For $u,v\in V(G)$, the \textit{$l^k$-interval} $I_{l^k}[u,v]$ consists of $u,v$ together with all vertices lying in some induced path between $u$ and $v$ whose length is at most $k$. A subset $S$ of vertices is called \textit{$l^k$-convex} if and only if for every pair $u,v \in S$ it holds that $I_{l^k}[u,v] \subseteq S$. Note that the empty set, the whole vertex set, and cliques (including singletons) are all $l^k$-convex sets. Let $\mathcal{C}$ be the family of all the $l^k$-convex sets of $V(G)$. It is not difficult to see that $(G,\mathcal{C})$ is a graph convexity space. The family $\mathcal{C}$ is called {\em $l^k$-convexity}.

For $W\subseteq V(G)$, we define $I_{l^k}[W]=\cup_{u,v\in W} I_{l^k}[u,v]$. Also, we define $I^j_{l^k}[W]$ recursively as follows: $I^0_{l^k}[W]=W$ and $I^j_{l^k}[W]=I[I^{j-1}_{l^k}[W]]$ for $j\geq 1$.

If $x\in I_{l^k}^m[W]$ for some $m\geq 0$, we say that $x$ is {\em captured} by $W$.

The \textit{$l^k$-convex hull} of $S \subseteq V(G)$, denoted by $\mathit{hull}_{l^k}(S)$, is the smallest set of vertices of $G$ that contains $S$ and is $l^k$-convex. Alternatively, it is the intersection of all $l^k$-convex sets of $G$ that contain $S$. It can be easily shown that $\mathit{hull}_{l^k}(S)=I^j_{l^k}[S]$ for some integer $j\geq 0$; in fact, $j$ can be taken as the minimum index for which $I^j_{l^k}[S]=I^{j+1}_{l^k}[S]$.


A vertex $x$ of an $l^k$-convex set $S\subseteq V(G)$ is an \textit{extreme point} of $S$ if $S\backslash\{x\}$ is also an $l^k$-convex set of $G$. The set of all the extreme points of $S$ is denoted by $\mathit{Ext}_{l^k}(S)$. Clearly, not every $l^k$-convex set contains extreme points. Also, note that a vertex $x$ is an extreme point of an $l^k$-convex set $S$ $(k\geq 2)$ if and only if $x$ is a simplicial vertex in $G[S]$.

Let $W$ be a subset of vertices of a graph $G$, and let $u$ be a vertex in $I_{l^2}[W]\setminus W$. Then there exist vertices $x,y\in W$ such that $xy\notin E(G)$ and $xu,uy\in E(G)$. This implies that $u$ is not a simplicial vertex in the subgraph induced by $I_{l^2}[W]$. In fact, it is not difficult to see that, for any $j$, the only simplicial vertices in the subgraph induced by $I^j_{l^2}[W]$ are those in $W$ (if any). This also holds for $I^j_{l^3}[W]$, $j\geq 0$. Thus:

\begin{proposition}\label{prop:simplicial}
Let $G$ be a graph, $W\subseteq V(G)$, and $k\in\{2,3\}$. Then if $u$ is an extreme point of $\mathit{hull}_{l^k}(W)$ then $u$ is a simplicial vertex in $G[W]$.
\end{proposition}

A graph $G$ is an {\em $l^k$-convex geometry} if, for every $l^k$-convex set $S$ of $G$, it holds that $\mathit{hull}_{l^k}(\mathit{Ext}_{l^k}(S))=S$.

For $a\in\{g,m,m^3,t,wt\}$, we similarly define {\em a-convexity}, {\em a-convex set}, $I^k_a(S)$, $\mathit{hull}_a(S)$, $\mathit{Ext}_a(S)$, and {\em $a$-convex geometry}, according to Table~\ref{tab:conv}.

\begin{table}[htbp]
\begin{center}
\begin{tabular}{|c|l|}
\hline
{\em symbol} & {\em associated path system}\\ \hline
$g$    & shortest (geodesic) paths~\cite{batten}\\
$m$    & induced (monophonic) paths~\cite{dourado-et-al,duchet}\\\
$m^3$  & induced paths of length at least three~\cite{dragan-et-al}\\
$t$    & tolled walks~\cite{alcon-et-al}\\
$wt$   & weakly tolled walks~\cite{gutierrez-tondato}\\ \hline
\end{tabular}
\caption{Some convexities associated with path systems}\label{tab:conv}

\end{center}
\end{table}

The subscript $a\in\{g,m,m^3,t,wt,l^k\}$ can be dropped from the notation when the path system under consideration is clear from the context.

Some important classes of graphs have been characterized as $a$-convex geometries, for $a$ in Table~\ref{tab:conv}. Chordal graphs and Ptolemaic graphs are the $m$- and $g$-convex geometries, respectively~\cite{farber-jamison}. Interval and proper interval graphs are the $t$- and $wt$-convex geometries, respectively~\cite{alcon-et-al,gutierrez-tondato}. Finally, weakly polarizable graphs are the $m^3$-convex geometries~\cite{dragan-et-al}.

All the above-mentioned classes of graphs are hereditary for induced subgraphs. As we shall see, this is not the case for $l^3$-convex geometries.

\section{A characterization of $l^2$-convex geometries}

In this section we prove that $l^2$-convex geometries are precisely the chordal $P_4$-free graphs, also referred as $C_4$-free cographs or trivially perfect graphs~\cite{golumbic} in the literature.

\begin{theorem}\label{thm:l2}
A graph $G$ is an $l^2$-convex geometry if and only if $G$ is a chordal $P_4$-free graph.
\end{theorem}

\begin{proof} Let $G$ be an $l^2$-convex geometry. Note that the extreme points of $V(G)$ are the simplicial vertices of $G$.

Suppose that $G$ contains an induced subgraph $C$ isomorphic to $C_n$, for $n>3$. It is clear that $\mathit{hull}(V(C))$ is a convex set of $G$, and so there exists a set $S \subseteq V(C)$ of extreme points of $\mathit{hull}(V(C))$ such that $\mathit{hull}(V(C))=\mathit{hull}(S)$. Since $C$ does not have simplicial vertices, by Proposition~\ref{prop:simplicial} the set $\mathit{hull}(V(C))$ does not have extreme points, a contradiction. Hence $G$ is a chordal graph.

Now, we prove that $G$ is a $P_4$-free graph. Let $P=x_0,\ldots,x_p$ be a maximum induced path of $G$, between two simplicial vertices of $G$. Such a path exists because $G$ is chordal~\cite{dirac}. Assume, in order to obtain a contradiction, that $p\geq 3$.

Note that $\mathit{hull}(V(P))$ is a convex set of $G$ with exactly two extreme points $x_0$ and $x_p$. Moreover if there exists $z\in \mathit{hull}(V(P))\setminus V(P)$, which is captured by $\{x_0,x_p\}$ with an induced path of length two between $x_0$ and $x_p$, $z$ must be adjacent to every $x_i$; otherwise, $P+x_0z+zx_p$ is an induced cycle of $G$ with at least four vertices. Since $x_0$ and $x_p$ are simplicial vertices of $G[\mathit{hull}(V(P))]$, $N[x_0]$ and $N[x_p]$ are cliques. Thus,

\[I_{l^2}[V(P)]=(N[x_0]\cap N[x_p])\cup\{x_0,x_p\}\]

\centerline{and}

\vspace{-0.6cm}

\[I^2_{l^2}[V(P)]=I_{l^2}[(N[x_0]\cap N[x_p])\cup\{x_0,x_p\}]=(N[x_0] \cap N[x_p])\cup\{x_0,x_p\}.\]

No vertex in $V(P)$ is captured by $N[x_0]\cap N[x_p]$, and so $\mathit{hull}(V(P))\neq \mathit{hull}(\{x_0,x_p\})$, a contradiction since $\mathit{hull}(V(P))$ is a convex set of a convex geometry. This concludes the first part of the proof.

\smallskip

Conversely, let $G$ be a chordal $P_4$-free graph, and let $S\subseteq V(G)$ be an $l^2$-convex set of $G$. Since $G[S]$ is a chordal graph, every vertex in $S$ lies in an induced path between two simplicial vertices of $G[S]$ (extreme points of $S$). Such a path must have length two, otherwise $G$ contains $P_n$ $(n\geq 4)$ as an induced subgraph. This means that $S$ is the convex hull of its extreme points, and the proof is complete.
\end{proof}

The following example shows that the statement of Theorem~\ref{thm:l2} is not true if we replace ``$l^2$-convex geometry/chordal $P_4$-free'' by ``$l^3$-convex geometry/chordal $P_5$-free''. Let $G$ be the chordal graph in Figure~\ref{figure-example}.


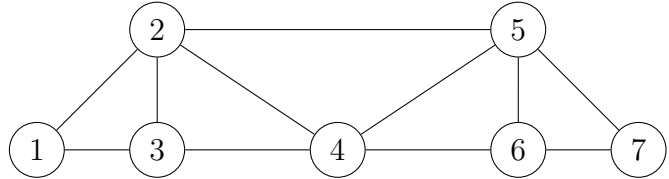
\begin{figure}[ht]
\begin{center}
\begin{tikzpicture}[scale=0.8]
    \node[draw,circle] (6) at (4,0) {$1$};
 \node[draw,circle] (7) at (6,0) {$3$};
  \node[draw,circle] (12) at (9,0) {$4$};
    \node[draw,circle] (8) at (12,0) {$6$};
   \node[draw,circle] (9) at (14,0) {$7$};
   \node[draw,circle] (10) at (6,2) {$2$};
   \node[draw,circle] (11) at (12,2) {$5$};
     \draw (6)--(7)--(12)--(8)--(9)--(11)--(8);
   \draw (12)--(10)--(11)--(12);
    \draw (7)--(10)--(6);
\end{tikzpicture}\end{center}
\caption{A chordal $l^3$-convex geometry that is not $P_5$-free.} \label{figure-example}
\end{figure}

Note that $\mathit{Ext}_{l^3}(V(G))=\{1,7\}$ but $I_{l^3}[\{1,7\}]=\{1,2,5,7\}\neq V(G)$. Two iterations are necessary for the set $\{1,7\}$ to capture all the vertices of $G$, i.e.,
\[I^2_{l^3}[\{1,7\})]=\mathit{hull}_{l^3}(\{1,7\})=V(G).\]

Up to symmetries, the nontrivial $l^3$-convex sets of $G$ are:
\[\{1,2,3,4\},\{1,2,3,4,5\},\{1,2,3,4,5,6\},\{2,3,4,5\},\{2,3,4,5,6\}.\]

It is a tedious but straightforward task to check that each set $S$ above satisfies $\mathit{hull}_{l^3}(\mathit{Ext}_{l^3}(S))=S$. Thus, $G$ is a chordal $l^3$-convex geometry. However, $G$ is not $P_5$-free. In order to characterize $l^3$-convex geometries, we need another properties, as we shall see in the next section.

\section{A characterization of $l^3$-convex geometries}

In this section we provide a characterization of $l^3$-convex geometries by means of a special property that must be satisfied by its induced gems with at least six vertices. Before stating the main result of this section, we need additional terminology.

Let $G$ be a graph. We say that an induced sugraph $H$ of $G$ is a {\em convex subgraph} if $V(H)$ is an $l^3$-convex subset of $G$. Let $P$ be an induced path and $x\in V(G)$ be a vertex; we say that {\em $P$ avoids $x$} if $x\notin V(P)$. Hereafter, an induced path with length three is simply referred as a $P_4$.

Let $G_n (n\geq 4)$ be an $n$-gem with vertices $x_0,\ldots,x_n,u_n$ as in Figure~\ref{fig0}, and assume that $G_n$ is an induced subgraph of $G$. We say that $G_n$ is {\em solved} if there exists in $G$ a $P_4$ connecting $x_0$ and $x_n$ that avoids $u_n$.


Before stating Theorem~\ref{thm:l3}, we need the following lemma:

\begin{lemma}\label{lem:diam3}
If $G$ is a graph with $\mathit{diam}(G)\leq 3$ then every connected $l^3$-convex subgraph $H$ of $G$ satisfies $\mathit{diam}(H)\leq 3$.
\end{lemma}

\begin{proof}
Suppose $\mathit{diam}(H)=d>3$ for some connected convex subgraph $H$ of $G$, and let $u,v\in V(H)$ such that $\mathit{dist}_H(u,v)=d>3$. Since $\mathit{dist}_G(u,v)\leq 3$, there must exist an induced path $Q$ in $G$ between $u$ and $v$ of length at most three. Clearly, there exists a vertex $x\in V(Q)$ such that $x\notin V(H)$. However, $x\in I[u,v]\subseteq V(H)$, because $H$ is a convex subgraph. This is a contradiction. Therefore, $\mathit{diam}(H)\leq 3$.
\end{proof}

\begin{theorem}\label{thm:l3}
A graph $G$ is an $l^3$-convex geometry if and only if the following conditions hold:

\begin{enumerate}

\item $G$ is chordal;

\item $\mathit{diam}(G)\leq 3$;

\item every induced $n$-gem $(n\geq 4)$ contained in $G$ is solved.
\end{enumerate}

\end{theorem}

\begin{proof}
Suppose that $G$ is an $l^3$-convex geometry. Recall that the extreme points of $V(G)$ are its simplicial vertices.


First, we prove that $G$ is a chordal graph. Assume by contradiction that $G$ contains an induced cycle with at least four vertices. It is clear that $\mathit{hull}(V(C))$ is a convex set of $G$, and, since $G$ is a convex geometry, there exists a set $S \subseteq V(C)$ of extreme points of $\mathit{hull}(V(C))$ such that $\mathit{hull}(V(C))=\mathit{hull}(S)$. Since $C$ does not contain simplicial vertices, by Proposition~\ref{prop:simplicial} the set $\mathit{hull}(V(C))$ does not contain extreme points as well, a contradiction. Hence $G$ is a chordal graph.

Now, we prove that $\mathit{diam}(G)\leq 3$. Assume, in order to obtain a contradiction, that $\mathit{diam}(G)>3$. Since $G$ is chordal, there exist two simplicial vertices $x_0,x_n\in V(G)$ such that $d(x_0,x_n)=\mathit{diam}(G)$. Let $P=x_0,x_1,\ldots,x_n$ be an induced path whose length is $\mathit{diam}(G)$. Note that $\mathit{hull}(V(P))$ is a convex set of $G$ with exactly two extreme points, $x_0$ and $x_n$. Since every induced path between $x_0$ and $x_n$ in $G$ is of length greater than or equal to four, $\{x_0,x_n\}$ cannot capture, with induced paths of length less than or equal to three, any vertex of $\mathit{hull}(V(P))$. Thus, $\mathit{hull}(V(P))\neq\mathit{hull}(\{x_0,x_n\})$, which is a contradiction since $\mathit{hull}(V(P))$ is a convex set of a convex geometry. Hence, $\mathit{diam}(G)\leq 3$.

Finally, suppose that $G$ contains $G_n$ $(n\geq 4)$, as in Figure~\ref{fig0}. Clearly, $\mathit{hull}(V(G_n))$ is a convex set of $G$. As $G$ is a convex geometry and $G_n$ contains exactly two simplicial vertices, by Proposition~\ref{prop:simplicial} we conclude that $\mathit{hull}(V(G_n))=\mathit{hull}(\{x_0,x_n\})$. Clearly, $u_n$ is captured by an induced path $P$ of length at most 3 between $x_0$ and $x_n$, namely $P=x_0,u_n,x_n$. Since $u_n$ is a neighbor of $x_0$ and $x_n$, no vertex $x_i$ for $i \neq 0,n$ can be captured by $\{x_0,x_n,u_n\}$. Moreover, no vertex $x\notin V(G_n)$ lying in an induced path of length two between $x_0$ and $x_n$ in $H=G[\mathit{hull}(V(G_n))]$ can be added to $\{x_0,x_n,u_n\}$ in order to capture a vertex $x_i$ $(i \neq 0,n)$, because $H$ is chordal and thus $x$ would necessarily be a neighbor of $u_n$ and of every $x_i$.

From the above arguments, there must exist at least one $P_4$ between $x_0$ and $x_n$, say $P'$, such that $\{x_0, x_n\}$ captures, with $P'$, some $x_i$ ($i\neq 0,n$) or vertices that will be used to capture the $x_i$'s in subsequent iterations of the computation of $\mathit{hull}(V(G_n))$. It is clear that $P'$ avoids $u_n$. Therefore, $G_n$ is solved, and this concludes the first part of the proof.

\smallskip

Conversely, suppose that conditions 1, 2, and 3 are satisfied. We have to prove that every convex set $S\subseteq V(G)$ is the convex hull of its extreme points.

First, we show that $S$ contains extreme points. Assume $|S|>1$ (otherwise, the proof is trivial). Clearly, the convex subgraph $H=G[S]$ is a chordal graph, and thus contains at least two simplicial vertices (i.e., extreme points of $S$). In addition, every vertex of $H$ lies in an induced path between two simplicial vertices.

Every vertex of $H$ lying in an induced path of length at most 3 between two simplicial vertices of $H$ trivially belongs to $\mathit{hull}(\mathit{Ext}(V(H)))$. Assume then there exists a vertex $x\in V(H)$ that does not lie in an induced path of length at most 3 between two simplicial vertices of $H$. We need to prove that, in this case, $x\in \mathit{hull}(\mathit{Ext}(V(H)))$ as well.

We consider two cases: $H$ contains $G_n$ $(n\geq 4)$ or not. The following claims deal with the two cases.

\begin{claim}\label{claim1} If $H$ does not contain $G_n$ $(n\geq 4)$ then $x$ lies in an induced path of $H$ of length at most four between two simplicial vertices of $H$.
\end{claim}

\begin{proof}
Suppose that every induced path of $H$ between simplicial vertices containing $x$ as an internal vertex is of length $n>4$, and let $P=x_0,\ldots,x_n$ be a minimal induced path satisfying this condition.

Note that $\mathit{hull}(V(P))$ is completely contained in $H$ and it is a convex set of $G$. On the other hand, its only extreme points are $x_0$ and $x_n$. Since $\mathit{hull}(V(P))$ is convex, by Lemma~\ref{lem:diam3} $G[\mathit{hull}(V(P))]$ has diameter less than or equal to three. Thus, $d_H(x_0,x_n)\leq 3$.


Suppose $d_H(x_0,x_n)=2$. Clearly, there exists an induced path $Q=x_0,u_n,x_n$ with $u_n$ not in $P$ (Figure \ref{fig1}).

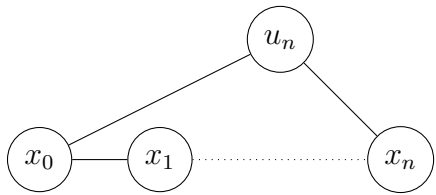
\begin{figure}[ht]
\tikzstyle{miEstilo}= [thin, dotted]
\begin{center}
\begin{tikzpicture}[scale=0.8]
\node[draw,circle] (1) at (-4,0) {$x_0$};
 \node[draw,circle] (2) at (-2,0) {$x_1$};
  \node[draw,circle] (3) at (2,0) {$x_n$};
   \node[draw,circle] (4) at (0,2) {$u_n$};
   \draw[miEstilo]  (2)--(3);
   \draw (1)--(4);
   \draw (4)--(3);
    \draw (1)--(2);
\end{tikzpicture}\end{center}
\caption{Induced path $Q$.} \label{fig1}
\end{figure}

Since $G[\mathit{hull}(V(P))]$ is chordal, $Q\cup P$ is not an induced cycle. Thus $u_n$ is adjacent to every $x_i$, $i=1,\ldots,n-1$. But then $Q\cup P=G_n (n>4)$, a contradiction.

Hence, $d_H(x_0,x_n)=3$. Let $Q'=x_0,y_1,y_2,x_n$ be an induced path of $G[\mathit{hull}(V(P))]$ with length three. Note that $|V(Q') \cap V(P)|\leq 3$ (Figure \ref{fig2}).

\begin{figure}[ht]
\tikzstyle{miEstilo}= [thin, dotted]
\begin{center}
\begin{tikzpicture}[scale=0.65]
\node[draw,circle] (1) at (-4,0) {$x_0$};
 \node[draw,circle] (2) at (-2,0) {};
  \node[draw,circle] (3) at (2,0) {};
   \node[draw,circle] (4) at (4,0) {$x_n$};
   \node[draw,circle] (5) at (0,2) {$y_1$};
   \draw[miEstilo]  (2)--(3);
   \draw (1)--(5)--(3);
   \draw (3)--(4);
    \draw (1)--(2);
    \node[draw,circle] (6) at (8,0) {$x_0$};
 \node[draw,circle] (7) at (10,0) {};
  \node[draw,circle] (8) at (12,0) {};
   \node[draw,circle] (9) at (14,0) {$x_n$};
   \node[draw,circle] (10) at (10,2) {$y_1$};
   \node[draw,circle] (11) at (12,2) {$y_2$};
   \draw[miEstilo]  (7)--(8);
   \draw (7)--(6)--(10)--(11);
   \draw (11)--(9);
    \draw (9)--(8);
\end{tikzpicture}\end{center}
\caption{Possible configurations for $Q'$ and $P$.} \label{fig2}
\end{figure}
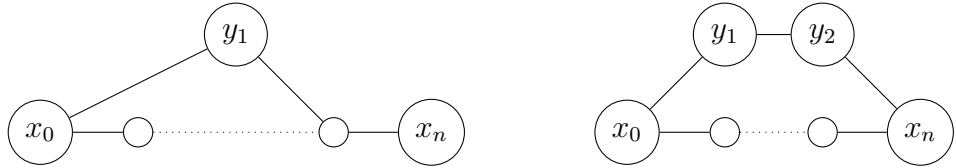

If $|V(Q')\cap V(P)|=3$, assume without loss of generality that $y_2=x_{n-1}$. Since $P$ and $Q'$ are induced paths and $G[\mathit{hull}(V(P))]$ is a chordal graph, it follows that $P[x_0,x_{n-1}] \cup Q'$ is not an induced cycle. Thus $y_1$ is adjacent to every $x_i$ lying in $P[x_0,x_{n-1}]\cup Q$. But then
$P[x_0,x_{n-1}]\cup Q'=G_{n-1}$ for $n-1>3$, a contradiction.

Now we analyze the case $|V(Q')\cap V(P)|=2$. Since $x_0$ and $x_n$ are simplicial vertices in $G[\mathit{hull}(V(P))]$, $x_1$ is adjacent to $y_1$ and $x_{n-1}$ is adjacent to $y_2$ (Figure \ref{fig3}). Let $Q''$ be the path $Q''=x_1,y_1,y_2,x_{n-1}$.

\begin{figure}[ht]
\tikzstyle{miEstilo}= [thin, dotted]
\begin{center}
\begin{tikzpicture}[scale=0.8]
    \node[draw,circle] (6) at (8,0) {$x_0$};
 \node[draw,circle] (7) at (10,0) {};
  \node[draw,circle] (8) at (12,0) {};
   \node[draw,circle] (9) at (14,0) {$x_n$};
   \node[draw,circle] (10) at (10,2) {$y_1$};
   \node[draw,circle] (11) at (12,2) {$y_2$};
   \draw[miEstilo]  (7)--(8);
   \draw (10)--(7)--(6)--(10)--(11);
   \draw (8)--(11)--(9);
    \draw (9)--(8);
\end{tikzpicture}\end{center}
\caption{Path $Q''$.} \label{fig3}
\end{figure}


Since $G[\mathit{hull}(V(P))]$ is a chordal graph, $P[x_1,x_{n-1}]\cup Q''$ is not an induced cycle. Then $y_1$ or $y_2$ is adjacent to $x_i$ for $i\notin\{1,n-1\}$. Moreover, there exists at least one index $l \in \{1,..,n-1\}$ such that $x_l$ is adjacent to $y_1$ and $y_2$. Let $i$ be the minimum index such that $y_2$ is adjacent to $x_i$, and $j$ be the maximum index such that $y_1$ is adjacent to $x_j$ (Figure \ref{fig4}). Clearly, $i\leq j$, otherwise $G[\{y_1,y_2,x_i,\ldots,x_j\}]$ is an induced cycle of size at least 4.

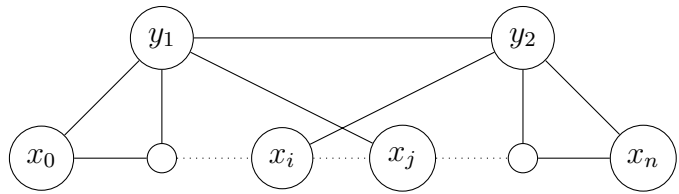
\begin{figure}[ht]
\tikzstyle{miEstilo}= [thin, dotted]
\begin{center}
\begin{tikzpicture}[scale=0.8]
    \node[draw,circle] (6) at (4,0) {$x_0$};
 \node[draw,circle] (7) at (6,0) {};
  \node[draw,circle] (12) at (8,0) {$x_i$};
   \node[draw,circle] (14) at (10,0) {$x_j$};
  \node[draw,circle] (8) at (12,0) {};
   \node[draw,circle] (9) at (14,0) {$x_n$};
   \node[draw,circle] (10) at (6,2) {$y_1$};
   \node[draw,circle] (11) at (12,2) {$y_2$};
   \draw[miEstilo]  (7)--(12)--(14)--(8);
   \draw (14)--(10)--(7)--(6)--(10)--(11)--(12);
   \draw (8)--(11)--(9);
    \draw (9)--(8);
\end{tikzpicture}\end{center}
\caption{Vertices $x_i$ and $x_j$.} \label{fig4}
\end{figure}

Note that $y_2$ is adjacent to $x_h$ for all $h\geq i$, and $y_1$ is adjacent to $x_m$ for all $m \leq j$, otherwise there exists an
induced cycle of size at least 4.

Note that $j<4$ and $i>n-4$, otherwise $G[\mathit{hull}(V(P))]$ would contain a $j$-gem $(j\geq 4)$ or an $(n-i)$-gem $(n-i\geq 4)$, which is impossible.

We show that $n<7$. Assume by contradiction that $n\geq 7$. Thus $n-4\geq 3$, and this implies $i>n-4\geq 3$. But this is impossible because $j<4$ and $i\leq j$. Hence $n<7$.

Assume $n=6$. Since $i>n-4$ and $j<4$, if $i\neq 3$ then $i\geq 4>j$, a contradiction since $i\leq j$. Thus, $i=3$. Analogously, $j=3$. But then $G[\{x_0,x_1,x_2,x_3,y_2,y_1\}]=G_4$, which is impossible.

Assume now $n=5$, and suppose that $j=i$. Since $i>n-4$ and $i\leq j$, it follows that $j=2$ or $j=3$, and then either $G[\{y_1,y_2,x_2,x_3,x_4,x_5\}]=G_4$ or $G[\{y_1,y_2,x_0,\ldots,x_3\}]=G_4$, a contradiction. If $i<j$ then, since $j<4$, $i=2$ and $j=3$. Thus $x$ is an internal vertex either of the induced path $x_0,x_1,x_2,y_2,x_5$ or the induced path  $x_0,y_1,x_3,x_4,x_5$, contradicting the choice of $P$.

Hence $n\leq 4$ and $x$ lies in an induced path of $H$ of length at most four between two simplicial vertices of $H$.
\end{proof}

\begin{claim}\label{claim2} If $H$ does not contain $G_n$ $(n\geq 4)$ then $x$ is captured by $\mathit{Ext}(V(H))$.
\end{claim}

\begin{proof}
By Claim \ref{claim1}, if $x\notin I[\mathit{Ext}(V(H))]$ then there exists in $H$ an induced path $P=x_0,\ldots,x,\ldots,x_4$ of length four such that $x_0$ and $x_4$ are simplicial vertices of $H$. Since $\mathit{hull}(V(P))\subseteq V(H)$ is a convex set of $G$, by Lemma~\ref{lem:diam3} the diameter of $G[\mathit{hull}(V(P))]$ is less than or equal to three. Also, it does not contain $G_n (n\geq 4)$, because it is an induced subgraph of $H$.

Note that the diameter of $G[\mathit{hull}(V(P))]$ is exactly three (since it is a chordal graph not containing $G_n \ (n\geq 4)$). In addition, $x_0$ and $x_4$ are the only simplicial vertices of $G[\mathit{hull}(V(P))]$.  Then  there exists in $G[\mathit{hull}(V(P))]$ an induced path $Q=x_0,y_1,y_2,x_4$ of length three. Observe that $P$ and $Q$ can have at most three vertices in common and the extreme points of $\mathit{hull}(V(P))$ are also extreme points of $H$.

If $|V(Q)\cap V(P)|=3$, without loss of generality assume that $y_2=x_3$. Since $x\notin I[\mathit{Ext}(V(H))]$, we have $x_3\neq x$. Note that $x_3\in I[\mathit{Ext}(\mathit{hull}(V(P)))]$ (since the length of the induced path $x_0,y_1,x_3,x_4$ is three).

On the other hand, $x$ is an internal vertex of the induced path $x_0,x_1,x_2,x_3$ and  $x_3\in I[\mathit{Ext}(\mathit{hull}(V(P))]$. Thus, $x\in I^2[\mathit{Ext}(\mathit{hull}(V(P))]$. Since the extreme points of $\mathit{hull}(V(P))$ are also extreme points of $H$, it follows that $x\in I^{m}[\mathit{Ext}(V(H))]$ with $m<4$.

Now suppose $|V(Q)\cap V(P)|=2$. Since $x_0$ and $x_4$ are simplicial vertices of $H$, we have that $y_1$ is adjacent to $x_1$ and $y_2$ is adjacent to $x_3$. On the other hand, $G[\mathit{hull}(V(P))]$ is a chordal graph, and thus $x_1,y_1,y_2,x_3,x_2,x_1$ is not an induced cycle. Thus, $y_1$ and $y_2$ are simultaneously adjacent to at least one vertex of $P$. By hypothesis, since $G_n$ ($n\geq 4$) is not an induced subgraph of $H$, both $y_1$ and $y_2$ are adjacent to $x_2$ (Figure \ref{fig5}).

\begin{figure}[ht]
\tikzstyle{miEstilo}= [thin, dotted]
\begin{center}
\begin{tikzpicture}[scale=0.8]
    \node[draw,circle] (6) at (4,0) {$x_0$};
 \node[draw,circle] (7) at (6,0) {$x_1$};
  \node[draw,circle] (12) at (9,0) {$x_2$};
    \node[draw,circle] (8) at (12,0) {$x_3$};
   \node[draw,circle] (9) at (14,0) {$x_4$};
   \node[draw,circle] (10) at (6,2) {$y_1$};
   \node[draw,circle] (11) at (12,2) {$y_2$};
   \draw[miEstilo]  (10)--(8);
   \draw[miEstilo]  (11)--(7);
   \draw (6)--(7)--(12)--(8)--(9)--(11)--(8);
   \draw (12)--(10)--(11)--(12);
    \draw (7)--(10)--(6);
\end{tikzpicture}\end{center}
\caption{$x_1$ may or may not be adjacent to $y_2$, and $x_3$ may or may not be adjacent to $y_1$.} \label{fig5}
\end{figure}
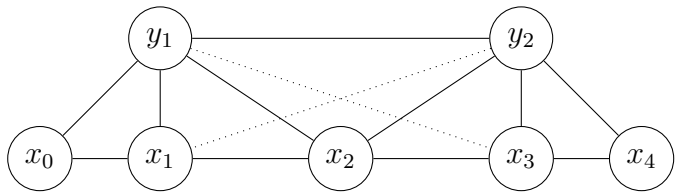

Note that $x_1$ may be adjacent to $y_2$, and $x_3$ may be adjacent to $y_1$. It is easy to see that $x$ is captured by $\{x_0,y_2\}$ with the induced path $x_0,x_1,x_2,y_2$ whenever $y_2$ is not adjacent to $x_1$, or by $\{x_4,y_1\}$ with the induced path $x_4,x_3,x_2,y_1$ if $y_1$ is not adjacent to $x_3$.

If $y_1$ is adjacent to $x_3$ or $y_2$ is adjacent to $x_1$ then $x$ is captured (depending on the position of $x$ in $P$) with one of the following induced paths: $x_0,x_1,y_2$; $x_4,x_3,y_1$; or $x_1,x_2,x_3$. In the last case, the following steps are needed: first, the set $\{x_0,x_4\}$ captures $y_1$ and $y_2$; next, $\{y_1, y_2\}$ captures $x_1$ and $x_3$; finally, $\{x_1,x_3\}$ captures $x=x_2$. Clearly,  $x\in I^{m}[\mathit{Ext}(V(P))]$ with $m=2$ or $m=3$. Hence, $x\in I^{m}[\mathit{Ext}(V(H))]$ with $m<4$. This completes the proof of the claim.

\end{proof}

\begin{claim}\label{claim3}
Suppose that $H$ contains $G_n$ $(n\geq 4)$, as in Figure~\ref{fig0}. Suppose also that the simplicial vertices of $G_n$ are captured by $\mathit{Ext}(V(H))$. Then every $x\in V(G_n)$ is captured by $\mathit{Ext}(V(H))$.
\end{claim}


\begin{proof}
Assume that $H$ contains $G_n$ $(n\geq 4)$ as in Figure~\ref{fig0}, and let $P$ be the induced path $P=x_0,x_1,\ldots,x_n$.

Note that $G[\mathit{hull}(V(G_n)]$ is a convex subgraph of $H$ with only two extreme points, $x_0$ and $x_n$. By hypothesis, $G_n$ is solved, and thus $G[\mathit{hull}(V(G_n))]$ contains an induced path $Q=x_0,y_1,y_2,x_n$ that avoids $u_n$.

Note that $P$ and $Q$ can have at most three vertices in common. Moreover, $y_1$ and $y_2$ are captured by $\{x_0,x_n\}$.

We prove the claim by induction on $n$.

\medskip

\noindent {\em Base case:} $n=4$

\medskip

If $Q=x_0,y_1,y_2,x_4$ and $P=x_0,x_1,x_2,x_3,x_4$ have three vertices in common, assume without loss of generality that  $y_2=x_3$. Note that $x_3$ is captured by $V(Q)$, and then every other vertex of $G_n$ is captured by $\{x_0,x_1,x_2,x_3\}$. Thus the claim follows in this case.

If $Q$ and $P$ have exactly two vertices in common then $y_1$ is adjacent to $x_1$ and $y_2$ is adjacent to $x_3$. In addition, both are adjacent to $u_4$. Since $G[\mathit{hull}(V(G_4))]$ is a chordal graph, $x_1,y_1,y_2,x_3,x_2,x_1$ is not an induced cycle, and since $P$ is an induced path there exists at least one vertex $x_l$, with $l\in\{1,2,3\}$, which is adjacent to both $y_1$ and $y_2$. See Figure~\ref{fig5}. Let $i,j\in\{1,2,3\}$ be such that $i$ is the minimum index for which $x_i$ is adjacent to $y_2$, and $j$ is the maximum index for which $x_j$ is adjacent to $y_1$. Note that $i\leq j$, otherwise
there exists an induced cycle of size at least four.

Since $n=4$, it is clear that $x_0,\ldots,x_i,y_2$ is an induced path of length at most three whenever $i\neq 3$. Similarly, $y_1,x_j,\ldots,x_4$ is an induced path of length at most three whenever $j\neq 1$.

If $i=3$ then $y_1$ is adjacent to $x_3$, and thus $\{y_1,x_4\}$ captures $x_3$. Next, $\{x_0,x_3\}$ captures the other vertices of $G_n$. We proceed analogously for $j=1$.

If $i\neq 3$ and $j\neq 1$, consider the following induced paths of length at
most three: $Q'=x_0,\ldots,x_i,y_2$ and $Q''=y_1,x_j,\ldots,x_4$. The existence of such paths show that $x_i$ and $x_j$ are captured by $\{x_0,y_2\}$ and $\{y_1,x_4\}$, respectively. Then the other vertices of $G_n$ are also captured, and the base case is complete.

\medskip

\noindent {\em Inductive step:} $n>4$

\medskip

Assume that, for every $l$-gem $G'$ with $4\leq l<n$ contained in $H$ whose simplicial vertices are captured by $\mathit{Ext}(V(H))$, all vertices of $G'$ are captured by $\mathit{Ext}(V(H))$.

By hypothesis, $G_n$ $(n>4)$ is solved, and then there exists an induced path $Q=x_0,y_1,y_2,x_n$ in $G[\mathit{hull}(V(G_n)]$ that avoids $u_n$.

If $Q$ and $P=x_0,x_1,\ldots,x_{n-1},x_n$ have three vertices in common, assume without loss of generality that $x_{n-1}\in V(Q)\cap V(P)$. In this case, $x_{n-1}$ is captured with the induced path $Q$. Let $G'$ be the $(n-1)$-gem induced by $V(G_n)\backslash\{u_n\}$. Since $x_0$ and $x_{n-1}$ are already captured and $x_{n-1}$ is a simplicial vertex of $G'$, by the induction hypothesis all the vertices of $G'$ are captured, and this implies that all the vertices of $G_n$ are captured by $\mathit{Ext}(V(H))$.

If $Q$ and $P$ have only two vertices in common, we have the following situation. Since $G[\mathit{hull}(V(G_n))]$ is a chordal graph, $x_1,y_1,y_2,$ $x_{n-1},x_{n-2},\ldots,x_2,x_1$ is not an induced cycle. Thus there exists at least one vertex $x_l$, with $l\in\{1,\ldots,n-1\}$, which is adjacent to $y_1$ and $y_2$. Let $i,j\in\{1,2,\ldots,n-1\}$ be such that $i$ is the minimum index for which $x_i$ is adjacent to $y_2$, and $j$ is the maximum index for which $x_j$ adjacent to $y_1$. Note that $i\leq j$. See Figure~\ref{fig6}).

\begin{figure}[ht]
\tikzstyle{miEstilo}= [thin, dotted]
\begin{center}
\begin{tikzpicture}[scale=0.8]
    \node[draw,circle] (6) at (4,0) {$x_0$};
 \node[draw,circle] (7) at (6,0) {$x_1$};
  \node[draw,circle] (12) at (8,0) {$x_i$};
    \node[draw,circle] (13) at (10,0) {$x_j$};
    \node[draw,circle] (8) at (12,0) {};
   \node[draw,circle] (9) at (14,0) {$x_n$};
   \node[draw,circle] (10) at (6,2) {$y_1$};
   \node[draw,circle] (11) at (12,2) {$y_2$};
    \draw (6)--(7);
  \draw[miEstilo] (7)--(12);
  \draw (9)--(11)--(8);
  \draw[miEstilo] (12)--(13);
  \draw (11)--(13);
  \draw[miEstilo] (13)--(8);
  \draw (8)--(9);
   \draw (12)--(10)--(11)--(12);
    \draw (7)--(10)--(6);
    \draw (10)--(13);
\end{tikzpicture}\end{center}
\caption{Inductive step: vertices $x_i$ and $x_j$.} \label{fig6}
\end{figure}
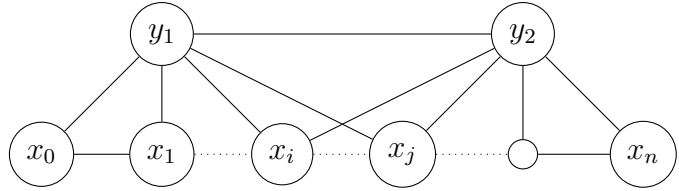

If $i=1$, $x_1$ is captured with the induced path $x_0,x_1,y_2$. Note that the vertices $x_1,x_2,\ldots,x_n, y_2$ induce an $(n-1)$-gem $G''$ whose simplicial vertices, $x_1$ and $x_n$, are already captured. By the induction hypothesis, all the vertices of $G''$, and thus of $G_n$, are captured by $\mathit{Ext}(V(H))$. We proceed analogously for $j=n-1$, by considering the $(n-1)$-gem induced by the vertices $y_1,x_0,\ldots,x_{n-1}$.

If $i>1$ and $j<n-1$, we analyze three cases:

\medskip

\noindent {\bf Case 1:} If $i>2$ and $j<n-2$, the gem induced by $\{x_0,x_1,\ldots,x_i,y_2,y_1\}$ (with simplicial vertices $x_0$ and $y_1$) and the gem induced by $\{y_1,y_2,x_j,\ldots,x_n\}$ (with simplicial vertices $x_n$ and $y_1$) are $l$-gems with $4\leq l<n$. Since the vertices $x_0,x_n,y_1,y_2$ have already been captured, by the induction hypothesis all vertices of such gems are captured in subsequent iterations. This means that vertices $x_1,\ldots,x_i$ and $x_j,\ldots,x_{n-1}$ are captured.

If $i=j$ then all vertices of $G_n$ are captured.

If $i<j$, either $Q'=x_i,x_{i+1},\ldots,x_{j-1},x_j$ is an induced path of length at most three, or the vertices of $Q'$ along with $u_n$ induce an $l$-gem $G'''$ ($4\leq l<n$) with simplicial vertices $x_i$ and $x_j$. In the latter case, by the induction hypothesis, all vertices of $G'''$ are captured. Thus, in either case, we conclude that all vertices of $G_n$ are captured.

\medskip

\noindent {\bf Case 2:} If $i=2$ and $j<n-2$ (or $i>2$ and $j=n-2$), vertices $x_1$ and $x_2$ are captured with the induced path $x_0,x_1,x_2,y_2$. Using similar arguments as above, the claim follows by considering the sets $S_1=\{y_1,y_2,x_j,\ldots,x_n\}$ and $S_2=\{x_2,x_3,\ldots,x_j\}$. In the former case, $S_1$ induces a gem with already-captured simplicial vertices $y_1$ and $x_n$. In the latter, $S_2$ either induces a $P_r$ with $r\leq 4$ or, along with $u_n$, an $l$-gem $(4\leq l<n)$ with already-captured simplicial vertices $x_2$ and $x_j$.

\medskip

\noindent {\bf Case 3:} If $i=2$ and $j=k-2$, we have that $x_1,x_2$ are captured with the induced path $x_0,x_1,x_2,y_2$, and $x_{n-2},x_{n-1}$ are captured with the induced path $y_1,x_{n-2},x_{n-1},x_n$.

If the induced path $x_2,x_3,\ldots,x_{n-2}$ has length at most three, vertices $x_3,\ldots,x_{n-3}$ are captured by $\{x_2,x_{n-2}\}$. Otherwise, the claim follows by considering the gem induced by $\{x_2,\ldots,x_{n-2},u_n\}$, where $x_2$ and $x_{n-2}$ are its already-captured simplicial vertices.

\end{proof}

Now, we conclude the proof that $V(H)=\mathit{hull}(\mathit{Ext}(V(H)))$. Recall that there is a vertex $x\in V(H)\setminus I(\mathit{Ext}(V(H)))$ such that $x$ lies in an induced path of length at least four $P=x_0,\ldots,x,\ldots,x_n$, where $x_0$ and $x_n$ are simplicial vertices of $H$. We need to prove that $x\in\mathit{hull}(\mathit{Ext}(V(H)))$.

Clearly, $\mathit{hull}(V(P))$ is contained in $V(H)$ and   $H'=G[\mathit{hull}(V(P))]$ is a chordal subgraph of $H$. Also, $H'$ is a convex subgraph of $G$, and thus by Lemma~\ref{lem:diam3} we have ${\mathit diam}(H')\leq 3$. Note that $x_0$ and $x_n$ are the only simplicial vertices of $H'$.

If ${\mathit diam}(H')=2$, there exists an induced path $x_0,u_n,x_n$. But then $u_n$ must be adjacent to all vertices of $P$, and the set $V(P)\cup\{u_n\}$ induces $G_n$ $(n\geq 4)$. By Claim~\ref{claim3}, all vertices of $G_n$ are captured by $\mathit{Ext}(V(H))$, and the proof follows.

If ${\mathit diam}(H')=3$, there exists an induced path $Q=x_0,y_1,y_2,x_n$ such that $Q$ and $P$ have at most three vertices in common.

If $|V(Q)\cap V(P)|=3$, assume $x_{n-1}\in V(Q)\cap V(P)$. Then $x_{n-1}$ is captured with $Q$. If $n=4$ then $x_1,x_2$ are captured by $\{x_0,x_3\}$, and the proof follows. If $n>4$, since $H'$ is chordal it follows that $y_1$ is adjacent to all vertices of $V(P)\setminus\{x_n\}$, and then $\{x_0,\ldots,x_{n-1},y_1\}$ induces an $l$-gem with $l\geq 4$ and already-captured simplicial vertices $x_0$ and $x_{n-1}$. Hence, by Claim~\ref{claim3}, the proof follows.

If $|V(Q)\cap V(P)|=2$, since $x_0$ and $x_n$ are simplicial vertices, we have that $y_1$ is adjacent to $x_1$ and $y_2$ is adjacent to $x_{n-1}$. Also, $x_1,y_1,y_2,x_{n-1},\ldots,x_1$ is not an induced cycle, and then there must exist $x_i$ simultaneously adjacent to $y_1$ and $y_2$. Following previous arguments, let $i,j \in \{1,2,\ldots,n-1\}$ be such that $i$ is the minimum index for which $x_i$ is adjacent to $y_2$, and $j$ the maximum
index for which $x_j$ adjacent to $y_1$. We have the situation illustrated in Figure~\ref{fig6}. The proof then follows by using the same argumentation as in Claim~\ref{claim3}.


Therefore, $V(H)=\mathit{hull}(\mathit{Ext}(V(H)))$ and $G$ is an $l^3$-convex geometry.

\end{proof}

\subsection{The class of $l^3$-convex geometries is not hereditary}
\label{sec:not-hered}

Observe that the graph $G$ depicted in Figure~\ref{figure-example} is a proper interval graph (and, thus, an interval graph and a chordal graph). Also, $G$ is weak polarizable (see~\cite{olariu}). In addition, for $a\in\{m,m^3,t,wt\}$, we have the following facts:

\begin{itemize}
\item $\mathit{\mathit{Ext}}_a(V(G))=\{1,7\}$;
\item $\mathit{\mathit{hull}}_a(\{1,7\})=I_a[\{1,7\}]=V(G)$;
\item $G-x$ is an $a$-convex geometry for every $x\in V(G)$.
\end{itemize}

\noindent However, for $x\in\{2,5\}$, $G'=G-x$ is not an $l^3$-convex geometry, since:

\begin{itemize}
\item $V(G')$ is trivially an $l^3$-convex set of $G'$,
\item $\mathit{\mathit{Ext}_{l^3}(V(G'))}=\{1,7\}$, and
\item $\mathit{\mathit{hull}}_{l^3}(\{1,7\})=\{1,7\}\neq V(G')$.
\end{itemize}

\noindent This shows that the class of $l^3$-convex geometries is not hereditary.

\section{Conclusions and future work}
\label{sec:conclu}

Characterizing $l^k$-convex geometries for $k\geq 4$ is an interesting open question. In order to tackle such a question, the following result will be useful:

\begin{proposition}
Let $k\geq 2$. If $G$ is an $l^k$-convex geometry then $G$ is a chordal graph with $\mathit{diam}(G)\leq k$.
\end{proposition}.

The proof of the above proposition uses the same arguments in the necessity proof of Theorem~\ref{thm:l3}.

Regarding complexity aspects, recognizing whether a graph $G$ is an $l^2$-convex geometry can be easily done in linear time by testing whether $G$ is a chordal cograph (see~\cite{corneil-et-al,rose-et-al}).

Finding, if possible, an efficient recognition algorithm for $l^3$-convex geometries amounts to finding an efficient test of condition 3 in Theorem~\ref{thm:l3}, for chordal graphs with diameter at most three.


\end{document}